\def\marker{\>\hbox{${\vcenter{\vbox{
    \hrule height 0.4pt\hbox{\vrule width 0.4pt height 6pt
    \kern6pt\vrule width 0.4pt}\hrule height 0.4pt}}}$}\>}

\documentclass[12pt]{article}

\usepackage[left=2cm,right=2cm,top=3cm,bottom=3cm]{geometry}

\usepackage{amsmath, amssymb, latexsym, amsthm}
\usepackage{fullpage}

\newtheorem{theorem}{Theorem} 
\newtheorem{theorem*}{Theorem} 

\newtheorem{lemma}[theorem]{Lemma}

\theoremstyle{definition}

\theoremstyle{remark}

{\end{enumerate}}

\usepackage[dvips]{graphicx}

\graphicspath{{counterfigs/}}

\newcommand{\CL}[1]{\left\lceil #1 \right\rceil}

\newcommand{\FL}[1]{\left\lfloor #1 \right\rfloor}

\title{Rainbow Matchings of Size $\delta(G)$ in Properly Edge-Colored Graphs}
\author{Jennifer Diemunsch\footnotemark[1]\,\,\footnotemark[3]\quad
Michael Ferrara\footnotemark[1]\,\,\footnotemark[4],\quad
Casey Moffatt\footnotemark[1],\quad\\
Florian Pfender\footnotemark[2],\quad
and Paul S.\ Wenger\footnotemark[1]}

\begin{document}

\maketitle

\begin{abstract}
A {\it rainbow matching} in an edge-colored graph is a matching in which all the edges have distinct colors.
Wang asked if there is a function $f(\delta)$ such that a properly edge-colored graph $G$ with minimum degree $\delta$ and order at least $f(\delta)$ must have a rainbow matching of size $\delta$.
We answer this question in the affirmative; $f(\delta) = 6.5\delta$ suffices.
Furthermore, the proof provides a $O(\delta(G)|V(G)|^2)$-time algorithm that generates such a matching.

{\bf Keywords:} Rainbow matching, properly edge-colored graphs
\end{abstract}

\renewcommand{\thefootnote}{\fnsymbol{footnote}}
\footnotetext[1]{Dept.\ of Mathematical and Statistical Sciences, Univ.\ of Colorado Denver, Denver, CO; email addresses
{\tt jennifer.diemunsch@ucdenver.edu}, {\tt michael.ferrara@ucdenver.edu}, {\tt casey.moffatt@ucdenver.edu},
{\tt paul.wenger@ucdenver.edu}.}
\footnotetext[2]{
Institut f$\ddot{\text u}$r Mathematik, Univ. Rostock, Rostock, Germany;
{\tt Florian.Pfender@uni-rostock.de}.}
\footnotetext[3]{Research supported in part by UCD GK12 Transforming Experiences Project, NSF award \# 0742434.}
\footnotetext[4]{Research supported in part by Simons Foundation Collaboration Grant \# 206692.}
\renewcommand{\thefootnote}{\arabic{footnote}}

\baselineskip18pt

\section{Introduction}

All graphs under consideration in this paper are simple, and we let $\delta(G)$ and $\Delta(G)$ denote the minimum and maximum degree of a graph $G$, respectively.  A {\it rainbow subgraph} in an edge-colored graph is a subgraph in which all edges have distinct colors.  Rainbow matchings are of particular interest given their connection to transversals of Latin squares: each Latin square can be converted to a properly edge-colored complete bipartite graph, and a transversal of the Latin square is a perfect rainbow matching in the graph.
Ryser's conjecture~\cite{Ryser} that every Latin square of odd order has a transversal can be seen as the beginning of the study of rainbow matchings.  Stein~\cite{Stein} later conjectured that every Latin square of order $n$ has a transversal of size $n-1$; equivalently every properly edge-colored $K_{n,n}$ has a rainbow matching of size $n-1$.  The connection between Latin transversals and rainbow matchings in $K_{n,n}$ has inspired additional interest in the study of rainbow matchings in triangle-free graphs.

Several results have been attained for rainbow matchings in arbitrarily edge-colored graphs.
The {\it color degree} of a vertex $v$ in an edge-colored graph $G$, written $\hat d(v)$, is the number of different colors on edges incident to $v$.
We let $\hat \delta (G)$ denote the minimum color degree among the vertices in $G$.
Wang and Li~\cite{WL} proved that every edge-colored graph $G$ contains a rainbow matching of size at least $\CL{\frac{5\hat\delta(G)-3}{12}}$, and conjectured that $\CL{\hat\delta(G)/2}$ could be guaranteed when $\hat \delta(G)\ge 4$.
LeSaulnier et al.~\cite{LSWW} then proved that every edge-colored graph $G$ contains a rainbow matching of size $\FL{\hat\delta(G)/2}$.
Finally, Kostochka and Yancey~\cite{KY} proved the conjecture of Wang and Li in full, and also that triangle-free graphs have rainbow matchings of size $\CL{\frac{2\hat\delta(G)}{3}}$.

Since the edge-colored graphs generated by Latin squares are properly edge-colored, it is of interest to consider rainbow matchings in properly edge-colored graphs.
In this direction, LeSaulnier et al.\ proved that a properly edge-colored graph $G$ satisfying $|V(G)|\neq \delta(G)+2$ that is not $K_4$ has a rainbow matching of size $\CL{\delta(G)/2}$.
Wang then asked if there is a function $f$ such that a properly edge-colored graph $G$ with minimum degree $\delta$ and order at least $f(\delta)$ must contain a rainbow matching of size $\delta$~\cite{Wang}.
As a first step towards answering this question, Wang showed that a graph $G$ with order at least $\frac{8\delta}{5}$ has a rainbow matching of size $\FL{\frac{3\delta(G)}{5}}$.

In this paper we answer Wang's question from \cite{Wang} in the affirmative.
\begin{theorem}\label{main}
If $G$ is a properly edge-colored graph satisfying $|V(G)|>\frac{13}{2}\delta-\frac{23}{2}+\frac{41}{8\delta}$, then $G$ contains a rainbow matching of size $\delta(G)$.
\end{theorem}
\noindent If $G$ is triangle-free, a smaller order suffices.
\begin{theorem}\label{trianglefree}
If $G$ is a triangle-free properly edge-colored graph satisfying $|V(G)|>\frac{49}{8}\delta-\frac{21}{2}+\frac{9}{2\delta}$, then $G$ contains a rainbow matching of size $\delta(G)$.
\end{theorem}

The proofs of Theorems~\ref{main} and~\ref{trianglefree} depend on the implementation of a greedy algorithm, a significantly different approach than those found in~\cite{KY},~\cite{LSWW},~\cite{Wang}, and~\cite{WL}.
This algorithm generates a rainbow matching in a properly edge-colored graph $G$ in $O(\delta(G)|V(G)|^2)$-time.

Since there are $n\times n$ Latin squares with no transversals (see~\cite{BR}) when $n$ is even, it is clear that $f(\delta)>2\delta$ when $\delta$ is even.
Furthermore, since maximum matchings in $K_{\delta,n-\delta}$ have only $\delta$ edges (provided $n\ge 2\delta$), there is no function for the order of $G$ depending on $\delta(G)$ that can guarantee a rainbow matching of size greater than $\delta$.

\section{Proof of the Main Results}

\begin{proof}[Proof of Theorem~\ref{main}]

We proceed by induction on $\delta(G)$. The result is trivial if $\delta(G)= 1$.  We assume that $G$ is a graph with minimum degree $\delta$ and order greater than $\frac{13}{2}\delta-\frac{23}{2}+\frac{41}{8\delta}$.

\begin{lemma}\label{maxdeg}
If $G$ satisfies $\Delta(G)>3\delta-3$, then $G$ has a rainbow matching of size $\delta$.
\end{lemma}

\begin{proof}
Let $v$ be a vertex of maximum degree in $G$.
By induction, there is a rainbow matching $M$ of size $\delta-1$ in $G-v$.
Since $v$ is incident to at least $3\delta-2$ edges with distinct colors, there is an edge incident to $v$ that is not incident to any edge in $M$ and also has a color that does not appear in $M$.
Thus there is a rainbow matching of size $\delta$ in $G$.
\end{proof}

\begin{lemma}\label{colorclass}
If $G$ has a color class containing at least $2\delta-1$ edges, then $G$ has a rainbow matching of size $\delta$.
\end{lemma}

\begin{proof}
Let $C$ be a color class with at least $2\delta-1$ edges.
By induction, there is a rainbow matching $M$ of size $\delta-1$ in $G-C$.
There are $2\delta-2$ vertices covered by the edges in $M$, thus one of the edges in $C$ has no endpoint covered by $M$, and the matching can be extended.
\end{proof}

The proof of Theorem~\ref{main} relies on the implementation of a greedy algorithm.
We begin by preprocessing the graph so that each edge is incident to at least one vertex with degree $\delta$.
To achieve this, we order the edges in $G$ and process them in order.
If both endpoints of an edge have degree greater than $\delta$ when it is processed, delete that edge.
In the resulting graph, every edge is incident to a vertex with degree $\delta$.
Furthermore, by Lemma~\ref{maxdeg} we may assume that $\Delta(G)\le 3\delta-3$; thus the degree sum of the endpoints of any edge is bounded above by $4\delta-3$.
After preprocessing, we begin the greedy algorithm.

In the $i$th step of the algorithm, a smallest color class is chosen (without loss of generality, color $i$), and then an edge $e_i$ of color $i$ is chosen such that the degree sum of the endpoints is minimum.
All the remaining edges of color $i$ and all edges incident with an endpoint of $e_i$ are deleted.
The algorithm terminates when there are no edges in the graph.

We assume that the algorithm fails to produce a matching of size $\delta$ in $G$; suppose that the rainbow matching $M$ generated by the algorithm has size $k$.
We let $R$ denote the set of vertices that are not covered by $M$.

Let $c_i$ denote the size of the smallest color class at step $i$.
Since at most two edges of color $i+1$ are deleted in step $i$ (one at each endpoint of $e_i$), we observe that $c_{i+1}+2\ge c_i$.
Otherwise, at step $i$ color class $i+1$ has fewer edges.
Let step $h$ be the last step in the algorithm in which a color class that does not appear in $M$ is completely removed from $G$.
It then follows that $c_h\le 2$, and in general $c_i\le 2(h-i+1)$ for $i\in [h]$.
Let $f_i$ denote the number of edges of color $i$ deleted in step $i$ with both endpoints in $R$.
Since $f_i< c_i$, we have $f_i\le 2(h-i)+1$ for $i\in[h]$.
Note that after step $h$, there are exactly $k-h$ colors remaining in $G$.
By Lemma~\ref{colorclass}, color classes contain at most $2\delta-2$ edges, and therefore the last $k-h$ steps remove at most $(k-h)(2\delta-2)$ edges.
Furthermore, for $i>h$, the degree sum of the endpoints of $e_i$ is at most $2(\delta-1)$.

For $i\in[h]$, let $x_i$ and $y_i$ be the endpoints of $e_i$, and let $d_i(v)$ denote the degree of a vertex $v$ at the beginning of step $i$.
Let $\mu_i = \max\{0,d_i(x_i)+d_i(y_i)-2\delta\}$; note that $2\delta\le 2\delta+\mu_i\le 4\delta-3$.
Thus, at step $i$, at most $2\delta+\mu_i+f_i-1$ edges are removed from the graph.
Since the algorithm removes every edge from the graph, we conclude that

\begin{equation}\label{ub}
|E(G)|\le (k-h)(2\delta-2)+\sum_{i=1}^{h} (2\delta+\mu_i+f_i-1).
\end{equation}

%

We now compute a lower bound for the number of edges in $G$.
Since the degree sum of the endpoints of $e_i$ in $G$ is at least $2\delta+\mu_i$, we immediately obtain the following inequality:
$$\frac{n\delta+\sum_{i\in[h]}\mu_i}{2}\le |E(G)|.$$

If $f_i>0$ and $\mu_i>0$, then there is an edge with color $i$ having both endpoints in $R$.
Since this edge was not chosen in step $i$ by the algorithm, the degree sum of its endpoints is at least $2\delta+\mu_i$, and one of its endpoints has degree at least $\delta+\mu_i$.
For each value of $i$ satisfying $f_i>0$, we wish to choose a representative vertex in $R$ with degree at least $\delta+\mu_i$.
Since there are $f_i$ edges with color $i$ with both endpoints in $R$, there are $f_i$ possible representatives for color $i$.
Since a vertex in $R$ with high degree may be the representative for multiple colors, we wish to select the largest system of distinct representatives.

Suppose that the largest system of distinct representatives has size $t$, and let $T$ be the set of indices of the colors that have representatives.
For each color $i\in T$ there is a distinct vertex in $R$ with degree at least $\delta+\mu_i$.
Thus we may increase the edge count of $G$ as follows:

\begin{equation}\label{lb}
\frac{n\delta+\sum_{i\in[h]}\mu_i+\sum_{i\in T}\mu_i}{2}\le |E(G)|.
\end{equation}

We let $\{f^\downarrow_i\}$ denote the sequence $\{f_i\}_{i\in[h]}$ sorted in nonincreasing order.
Since $f_i\le 2(h-i)+1$, we conclude that $f^\downarrow_i\le 2(h-i)+1$.
Because there is no system of distinct representatives of size $t+1$, the sequence $\{f^\downarrow_i\}$ cannot majorize the sequence $\{t+1,t,t-1,\ldots,1\}$.
Hence there is a smallest value $p\in[t+1]$ such that $f^\downarrow_p\le t+1-p$.
Therefore, the maximum value of $\sum_{i=1}^{h}f^\downarrow_i$ is bounded by the sum of the sequence $\{2h-1,2h-3,\ldots,2(h-p)+3,t+1-p,\ldots,t+1-p\}$.
Summing we attain
$$\sum_{i\in [h]}f_i\le (p-1)(2h-p+1)+(h-p+1)(t+1-p).$$
Over $p$, this value is maximized when $p=t+1$, yielding $\sum_{i\in[h]}f_i\le t(2h-t)$.
Since $h\le \delta-1$, we then have $\sum_{i\in[h]}f_i\le 2(\delta-1)t-t^2$.

We now combine bounds~\eqref{ub} and~\eqref{lb}:

$$\frac{n\delta+\sum_{i\in[h]}\mu_i+\sum_{i\in T}\mu_i}{2} \le (k-h)(2\delta-2)+\sum_{i=1}^{h} (2\delta+\mu_i+f_i-1).$$

\noindent Hence, since $k \leq \delta-1$,
\begin{align*}
\frac{n\delta}{2} &\le  (2\delta-1)(\delta-1)+\frac{1}{2}\sum_{[h]\setminus T}\mu_i + \sum_{i\in[h]}f_i\\
&\le (2\delta-1)(\delta-1)+(\delta-1-t)(\delta-3/2)+2(\delta-1)t-t^2\\
&\le 3\delta^2-\frac{11}{2}\delta+\frac{5}{2}+\left(\delta-\frac{1}{2}\right)t-t^2.
\end{align*}

\noindent This bound is maximized when $t=(\delta-\frac{1}{2})/2$.
Thus
$$n \le \frac{13}{2}\delta-\frac{23}{2}+\frac{41}{8\delta},$$
contradicting our choice for the order of $G$.
\end{proof}

\begin{proof}[Sketch of Proof of Theorem~\ref{trianglefree}]
When $G$ is triangle-free, Lemma~\ref{maxdeg} can be improved.
In particular, $\Delta(G)\le 2\delta-2$ since there is at most one edge joining a vertex of maximum degree to each edge in a matching of size $\delta-1$.
Since $\Delta(G)$ is used to bound the value of $\mu_i$ in the proof of Theorem~\ref{main}, the same argument yields the following inequality:
\begin{align*}
\frac{n\delta}{2} &\le
(2\delta-1)(\delta-1)+\frac{1}{2}\sum_{[h]\setminus T}\mu_i + \sum_{i\in[h]}f_i\\
&\le (2\delta-1)(\delta-1)+\frac{1}{2}(\delta-1-t)(\delta-2)+2(\delta-1)t-t^2\\
&\le \frac{5}{2}\delta^2-\frac{9}{2}\delta+2+\left(\frac{3}{2}\delta-1\right)t-t^2.
\end{align*}
This upper bound is maximized when $t=(\frac{3}{2}\delta-1)/2$, yielding
$$n\le \frac{49}{8}\delta-\frac{21}{2}+\frac{9}{2\delta}.$$
\end{proof}

\section{Conclusion}
The proof of Theorem~\ref{main} provides the framework of a $O(\delta(G)|V(G)|^2)$-time algorithm that generates a rainbow matching of size $\delta(G)$ in a properly edge-colored graph $G$.
Given such a $G$, we create a sequence of graphs $\{G_i\}$ as follows, letting $G=G_0$, $\delta=\delta(G)$, and $n=|V(G)|$.
First, determine $\delta(G_i)$, $\Delta(G_i)$, and the maximum size of a color class in $G_i$; this process takes $O(n^2)$-time.
If $\Delta(G_i)\le 3\delta(G_i)-3$ and the maximum color class has at most $2\delta(G_i)-2$ edges, then terminate the sequence and set $G_i=G'$.
If $\Delta(G_i)>3\delta(G_i)-3$, then delete a vertex $v$ of maximum degree and then process the edges of $G_i-v$, iteratively deleting those with two endpoints of degree at least $\delta(G_i)$; the resulting graph is $G_{i+1}$.
If $\Delta(G_i)\le 3\delta(G_i)-3$ but a maximum color class $C$ has at least $2\delta(G_i)-1$ edges, then delete $C$ and then process the edges of $G_i-C$, iteratively deleting those with two endpoints of degree at least $\delta(G_i)$; the resulting graph is $G_{i+1}$.
Note that $\delta(G_{i+1}) = \delta(G_i)-1$.
If this process generates $G_{\delta}$, we set $G'=G_\delta$ and terminate.
Generating the sequence $\{G_i\}$ consists of at most $\delta$ steps, each taking $O(n^2)$-time.

Given that $G'=G_i$, the algorithm from the proof of Theorem~\ref{main} takes $O(\delta n^2)$-time to generate a matching of size $\delta-i$ in $G'$.
The preprocessing step and the process of determining a smallest color class and choosing an edge in that class whose endpoints have minimum degree sum both take $O(n^2)$-time.
This process is repeated at most $\delta$ times.

A matching of size $\delta-(i+1)$ in $G_{i+1}$ is easily extended in $G_i$ to a matching of size $\delta-i$ using the vertex of maximum degree or maximum color class.
The process of extending the matching takes $O(\delta)$-time.
Thus the total run-time of the algorithm generating the rainbow matching of size $\delta$ in $G$ is $O(\delta n^2)$.\\

It is worth noting that the analysis of the greedy algorithm used in the proof of Theorem~\ref{main} could be improved.
In particular, the bound $c_{i+1}\ge c_i-2$ is sharp only if at step $i$ there are an equal number of edges of color $i$ and $i+1$ and both endpoints of $e_i$ are incident to edges with color $i+1$.
However, since one of the endpoints of $e_i$ has degree at most $\delta$, at most $\delta-1$ color classes can lose two edges in step $i$.
Since the maximum size of a color class in $G$ is at most $2\delta-2$, if $G$ has order at least $6\delta$, then there are at least $3\delta/2$ color classes.
Thus, for small values of $i$, the bound $c_i\le 2(k-i+1)$ can likely be improved.
However, we doubt that such analysis of this algorithm can be improved to yield a bound on $|V(G)|$ better than $6\delta$.

\end{document}